\documentclass{amsart}
 \usepackage{amssymb,amsmath,amsfonts,epsfig,latexsym,tikz}
\usepackage{enumerate}
\usepackage{mathtools}
\usepackage{subcaption}

\usetikzlibrary{positioning}
\usetikzlibrary{matrix}
\usetikzlibrary{decorations}
\usetikzlibrary{decorations.pathreplacing, decorations.pathmorphing, angles,quotes}
 
 \newtheorem{theorem}{Theorem}[section]
\newtheorem{definition}[theorem]{Definition}
\newtheorem{proposition}[theorem]{Proposition}
\newtheorem{lemma}[theorem]{Lemma}
\newtheorem{corollary}[theorem]{Corollary}

\theoremstyle{definition}
\newtheorem{remark}[theorem]{Remark}

\def\R{\mathbb{R}}

\def\<{\langle}
\def\>{\rangle}

\DeclareMathOperator{\Cl}{Cl}
\DeclareMathOperator{\codim}{codim}
\newcommand{\Conv}[1]{\operatorname{Conv}\left\{{#1}\right\}}

\DeclareMathOperator{\lk}{lk}

\DeclareMathOperator{\Star}{Star}

\DeclarePairedDelimiter{\set}{\{}{\}}
\newcommand\nullset\varnothing

\newcommand{\proofsection}[2]{
\vspace{0.1cm}
\noindent\textbf{{#1}} 
\textit{{#2}}
\vspace{0.1cm}
\par
}

\begin{document}

\title{Triangulations of simplices with vanishing local $h$-polynomial}           
    
\author[A. de Moura, E. Gunther, S. Payne, J. Schuchardt, and A. Stapledon]{Andr\'e de Moura, Elijah Gunther, Sam Payne, Jason Schuchardt, and Alan Stapledon}

\address{Viela do Mato 4 BL A RC Esq, Quinta da Beloura, 2710-695, Sintra, Portugal}
\email{andros.moura@gmail.com}

\address{UPenn Dept of Mathematics, David Rittenhouse Lab, 209 South 33rd Street
Philadelphia, PA 19104}
\email{elijahg@sas.upenn.edu}

\address{UCLA Dept of Mathematics, Math Sciences Building 6363, 520 Portola Plaza, Los Angeles, CA 90095}
\email{jasonsch@g.ucla.edu}

\address{UT Department of Mathematics, 2515 Speedway, PMA 8.100, Austin, TX 78712}
\email{sampayne@utexas.edu}

\address{}
\email{astapldn@gmail.com}

\begin{abstract}
Motivated by connections to intersection homology of toric morphisms, the motivic monodromy conjecture, and a question of Stanley, we study the structure of geometric triangulations of simplices whose local $h$-polynomial vanishes.  As a first step, we identify a class of refinements that preserve the local $h$-polynomial.  In dimensions 2 and 3, we show that all geometric triangulations with vanishing local $h$-polynomial are obtained from one or two simple examples by a sequence of such refinements.  In higher dimensions, we prove some partial results and give further examples.
\end{abstract}
 
\maketitle

\section{Introduction}

Let $\Gamma$ be a triangulation of a simplex $\Delta$ of dimension $d-1$. 
The $h$-polynomial $h(\Gamma; x) = h_0 + h_1 x + \cdots + h_d x^d$ is a common and
convenient way of encoding the number of faces of $\Gamma$ in each dimension. 
It is characterized by the equation
\[
\sum_{i=0}^{d} h_i (x+1)^{d-i} = \sum_{i=0}^d f_{i-1} x^{d-i},
\]
where $f_{-1} = 1$ and $f_i$ is the number of $i$-dimensional faces of $\Gamma$,
for $i \geq 0$.  The coefficients $h_i$ are non-negative integers. 
One powerful tool for studying $h(\Gamma;x)$ is the local $h$-polynomial
$\ell(\Gamma;x) = \ell_0 + \ell_1x + \cdots + \ell_d x^d$,
introduced by Stanley in his seminal paper \cite{Stanley92}. 
It is characterized via M\"obius inversion by the equation
\begin{equation} \label{eq:h=sum-of-l}
h(\Gamma;x) = \sum_{F \leq \Delta} \ell(\Gamma_F;x),
\end{equation}
where $\Gamma_F$ denotes the restriction of the triangulation $\Gamma$ to a face $F$
(which may be empty or all of $\Delta$), together with the condition $\ell(\varnothing; x) = 1$.

The local $h$-polynomial has remarkable properties. 
In particular, the coefficients $\ell_i$ are nonnegative and satisfy $ \ell_i = \ell_{d-i}$.
Moreover, if the subdivision is regular, then these coefficients are unimodal. 
Among other applications, Stanley used local $h$-polynomials to prove that $h$-polynomials
increase coefficientwise under refinement.

As discussed in Section 2, the local $h$-polynomial also behaves predictably with respect 
to basic operations on subdivisions.  It is additive for refinements that nontrivially 
subdivide only one facet, multiplicative for joins, and vanishes on the trivial subdivision. 
In particular, if $\Gamma'$ is a refinement of $\Gamma$ that nontrivially subdivides only
one facet, and that subdivision is the cone over a subdivision of a codimension 1 face, 
then $\ell(\Gamma';x) = \ell(\Gamma;x)$.  We call such subdivisions 
\emph{conical facet refinements}.  

Our aim is to study geometric triangulations with vanishing local $h$-polynomial.
One might hope that all such subdivisions are obtained from the trivial subdivision by a 
sequence of conical facet refinements. 
However, not all triangulations $\Gamma$ with vanishing local $h$-polynomial can be obtained from 
the trivial subdivision in this way.  One notable example
in dimension 2 is the following subdivision, which we call the 
\emph{triforce}\footnote{This name reflects the subdivision's realization in a sacred
golden relic that is the ultimate source of power in the action-adventure video game 
series \emph{The Legend of Zelda}.}.

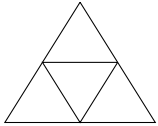
\medskip

\begin{center}
\begin{tikzpicture}[scale=2]
\draw (0.5,0.8)  -- (1,0) -- (0,0) -- (0.5,0.8);
\draw (0.75,0.4) -- (0.5,0) -- (0.25,0.4) -- (0.75,0.4);
\end{tikzpicture}
\end{center}

All of our theorems are for \emph{geometric triangulations}, and all triangulations in this paper are assumed to be geometric.

\begin{theorem} \label{thm:dim2}
In dimension 2, any triangulation with vanishing local $h$-polynomial is obtained from either the trivial subdivision or the triforce by a sequence of conical facet refinements.
\end{theorem}

\noindent The two cases in the theorem are distinguished by the \emph{internal edge graph}, i.e., the union of the edges that meet the interior of the subdivided simplex, which also figures prominently in the proof.  For an iterated conical facet refinement of the triforce subdivision, the internal edge graph has Euler characteristic zero and contains no vertices of the original triangle.  For an iterated conical facet refinement of the trivial subdivision in dimension 2, the internal edge graph is a tree that contains exactly one of the vertices of the original triangle.

In dimension three, the structural classification is even simpler.

\begin{theorem} \label{thm:dim3}
In dimension 3, any triangulation with vanishing local $h$-polynomial is obtained from the trivial subdivision  by a sequence of conical facet refinements.
\end{theorem}

\noindent The proof again relies on an analysis of the internal edge graph, which in dimension three is a union of trees, each of which contains exactly one vertex supported on a face of codimension at least 2 in the original simplex.

This structure of the internal edge graph is similar in higher dimensions.  Note that $\ell_0 = 0$ for any triangulation of a nonempty simplex.

\begin{theorem}  \label{thm:internaledgegraph}
Let $\Gamma$ be a triangulation of a simplex of dimension at least 3 such that $\ell_1 = \ell_2 = 0$.  Then the internal edge graph of $\Gamma$ is a union of trees each of which contains exactly one vertex supported on a face of codimension at least 2.
\end{theorem}

\noindent However, the pattern of obtaining all triangulations whose local $h$-polynomials vanish from a finite collection of examples by iterated conical facet refinements does not continue in higher dimensions.  See Section~\ref{sec:higherdim}.

Our investigation into the structure of triangulations with vanishing local $h$-polynomial
is motivated by recent connections to algebraic and arithmetic geometry. 
Local $h$-polynomials appear prominently in formulas for dimensions of homology groups of
intersection complexes for toric morphisms \cite{deCataldoMiglioriniMustata18} and 
multiplicities of eigenvalues of monodromy \cite{KatzStapledon16, Stapledon17}.  
In Igusa's $p$-adic monodromy conjecture \cite{Igusa75}, and the motivic generalization of
Denef and Loeser \cite{DenefLoeser98}, the essential question is understanding whether or
not these multiplicities vanish.  This problem is also natural and interesting from a
purely combinatorial viewpoint; Stanley specifically asked for a nice characterization of 
such triangulations in his original paper \cite[Problem~4.13]{Stanley92}.  

\medskip

\noindent \textbf{Acknowledgments.}  This project was carried out as part of the SUMRY undergraduate research program. We thank the other participants and mentors in this program for their support and encouragement.  We are also grateful to Matt Larson for helpful conversation related to this project.  Our work was partially supported by NSF grant DMS-1702428.

\section{Preliminaries}

We consider only geometric subdivisions $\Gamma$ of a $(d-1)$-simplex $\Delta$, except where explicitly stated otherwise.  In particular, all of the triangulations that we consider are realized by subdividing a linearly embedded simplex into subsimplices.  For further details and background on local $h$-polynomials, we refer the reader to the recent survey article of Athanasiadis \cite{Athanasiadis16}, as well as \cite{Stanley92}.

\subsection{Formulas for the local $h$-polynomial}
We recall two useful formulas for the local $h$-polynomial $\ell(\Gamma;x)$.  First, by applying M\"{o}bius inversion to \eqref{eq:h=sum-of-l}, we can express $\ell(\Gamma;x)$ as an alternating sum of $h$-polynomials:
\begin{equation} \label{eq:alternating sum}
\ell(\Gamma;x) = \sum_{F \leq \Delta}(-1)^{\codim (F)} h(\Gamma_F;x).
\end{equation}
Here, the codimension is $\codim(F) = d - 1 - \dim(F)$, since $\dim(\Delta) = d-1$.

Let $\sigma: \Gamma \rightarrow \Delta$ be the map taking a face $G \in \Gamma$ to the smallest face $F \leq \Delta$ that contains it.  One says that $G$ is \emph{carried by} $\sigma(G)$, and the \emph{excess} of $G$ is
\[
e(G) := \dim (\sigma(G)) - \dim (G).
\]
As in \cite[Prop. 2.2]{Stanley92}, we can then express the local $h$-polynomial as
\begin{equation}\label{eq:excess formula}
\ell(\Gamma; x) = \sum_{G \in \Gamma} (-1)^{\codim(G)} x^{d- e(G)} (x-1)^{e(G)}.
\end{equation}

\begin{definition}
We say that $G \in \Gamma$ is an \emph{interior face} if $\sigma(G) = \Delta$.  
\end{definition}

\subsection{Elementary properties of the local $h$-polynomial}

The following basic properties of local $h$-polynomials were mentioned in the introduction and will be used in our main arguments.

Given triangulations $\Gamma$ and $\Gamma'$ of simplices $\Delta$ and $\Delta'$, respectively, the join $\Gamma * \Gamma'$ is naturally a triangulation of the simplex $\Delta * \Delta'$.  It is well-known that $h$-polynomials are multiplicative for joins, i.e.,
\[
h(\Gamma*\Gamma';x) = h(\Gamma;x) \cdot h(\Gamma';x).
\]
It follows that the local $h$-polynomial has the analogous property \cite[Lemma~2.2]{AthanasiadisSavvidou12}, i.e.,
\begin{equation} \label{eq:join}
\ell(\Gamma*\Gamma';x) = \ell(\Gamma;x) \cdot \ell(\Gamma';x).
\end{equation}

A triangulation $\Gamma'$ of $\Delta$ is a \emph{refinement} of $\Gamma$ if every $G \in \Gamma$ is a union of faces of $\Gamma'$.  We then write $\Gamma'_G$ for the restriction of $\Gamma'$ to $G$. We are particularly interested in the restriction of $\Gamma'$ to maximal faces, or \emph{facets}.

\begin{definition}
Let $\Gamma'$ be a refinement of $\Gamma$, and let $G$ be a facet of $\Gamma$.  We say that $\Gamma'$ is a \emph{facet refinement of $\Gamma$ along $G$} if, for any facet $H \neq G$, the restriction $\Gamma'_H$ is trivial.
\end{definition}

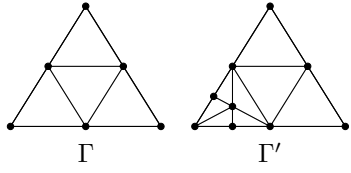
\begin{figure}[h!] \label{fig:facet-refinement}
{
\centering
\begin{tabular}{cc}
\begin{tikzpicture}[scale=2]
\draw[fill=white] (0,0) -- (1,0) -- (.5, .8) -- cycle;
\filldraw (0,0) circle (.6pt) -- (1,0) circle (.6pt) -- (.5, .8) circle (.6pt) -- (0,0);
\filldraw (.5,0) circle (.6pt) -- (.75, .4) circle (.6pt) -- (.25, .4) circle (.6pt) -- (.5,0);
\end{tikzpicture} 
&
\begin{tikzpicture}[scale=2]
\draw[fill=white] (0,0) -- (1,0) -- (.5, .8) -- cycle;
\filldraw (0,0) circle (.6pt) -- (1,0) circle (.6pt) -- (.5, .8) circle (.6pt) -- (0,0);
\filldraw (.5,0) circle (.6pt) -- (.75, .4) circle (.6pt) -- (.25, .4) circle (.6pt) -- (.5,0);
\filldraw (.25, .4) -- (.25, 0) circle (.6pt);
\filldraw (.125, .2) circle (.6pt) -- (.5, 0);
\filldraw (0,0) -- (.25, .133) circle (.6pt);
\end{tikzpicture} 
\\
$\Gamma$ 
&
$\Gamma'$
\end{tabular}\\
}
\caption{A facet refinement of the triforce along its lower left facet.}
\end{figure}

Local $h$-polynomials behave additively with respect to facet refinements. 

\begin{proposition}\label{additivity of local h}
Let $\Gamma'$ be a facet refinement of $\Gamma$ along $G$, and let $\Gamma'_{G}$ denote the triangulation of $G$ induced by restricting $\Gamma'$.  Then
\[
\ell(\Gamma';x) = \ell(\Gamma;x)  + \ell(\Gamma'_G; x).
\]

\end{proposition}

\noindent This follows from \cite[Corollary~4.7]{KatzStapledon16}, a much more general result about local $h$-polynomials for compositions of strong formal subdivisions, applied to $\Gamma' \to \Gamma \to \Delta$.  For the reader's convenience, we include a direct proof in our setting.

\begin{proof}
The assertion is vacuously true in dimension zero.  We proceed by induction on dimension.  First, we observe that
\begin{equation}
\label{eq:additiveh} h(\Gamma';x) = h(\Gamma;x) + h(\Gamma'_G;x) - h(G;x).
\end{equation}
This follows from the formula $\sum_{i = 0}^d f_i (x-1)^{d-i} = \sum_{i=0}^d h_i x^{d-i}$, by inclusion-exclusion. 

Next, write each $h$-polynomial in \eqref{eq:additiveh} as a sum of local $h$-polynomials, using \eqref{eq:h=sum-of-l}. The contributions of the empty faces cancel, and each nonempty face of $G$ contributes 0, since it is trivially subdivided. Therefore, we have
\begin{equation} \label{eq:sumofls}
\sum_{\{\varnothing\} \neq H \leq \Delta} \ell(\Gamma'_H;x) = \sum_{\{\varnothing\} \neq H \leq \Delta} \ell(\Gamma_H;x) + \sum_{\{\varnothing\} \neq H' \leq G} \ell(\Gamma'_{H'};x).
\end{equation}
For each nonempty face $H \leq \Delta$, either $\Gamma'_H = \Gamma_H$ or $\Gamma'_H$ is a facet refinement of $\Gamma_H$ along $G \cap H$.  If $H$ is a proper face of $\Delta$ then, by induction on dimension, we may assume $\ell(\Gamma'_H;x) = \ell(\Gamma_H;x) + \ell(\Gamma'_{G \cap H};x)$.  Similarly, for any nonempty face $H' \leq G$, either $\Gamma'_{H'}$ is the trivial subdivision, or $H'$ is carried by a face $H \leq \Delta$ such that  $H' = G \cap H$ and $\Gamma'_H$ is a facet refinement of $\Gamma_H$ along $H'$ . Thus, the contributions of proper faces in \eqref{eq:sumofls} cancel, and we conclude that $\ell(\Gamma';x) = \ell(\Gamma;x) + \ell(\Gamma'_G;x)$.
\end{proof}

\begin{definition}
A facet refinement $\Gamma'$ of $\Gamma$ along $G$ is a \emph{conical facet refinement} if $\Gamma'_G$ is the cone over $\Gamma'_H$ for some codimension 1 face $H < G$.
\end{definition}

\begin{corollary}  \label{cor:nfacetjoin}
Let $\Gamma'$ be a facet refinement of $\Gamma$ along $G$, and suppose that $\Gamma'_G$ is the join of two triangulations of faces of $G$, one of which has vanishing local $h$-polynomial.  Then $\ell(\Gamma';x) = \ell(\Gamma;x)$.  In particular, if $\Gamma'$ is a conical facet refinement of $\Gamma$, then $\ell(\Gamma';x) = \ell(\Gamma;x)$.  
\end{corollary}

\begin{proof}
This follows immediately from \eqref{eq:join} and Proposition~\ref{additivity of local h}.
\end{proof}

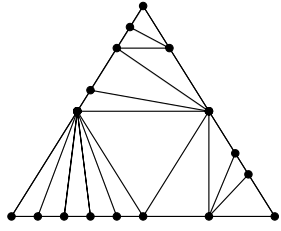
\begin{figure}[h!] \label{fig:conical-facet-refinement}
{
\centering
\begin{tikzpicture}[scale=3.5]
\draw[fill=white] (0,0) -- (1,0) -- (.5, .8) -- cycle;
\filldraw (0,0) circle (.4pt) -- (1,0) circle (.4pt) -- (.5, .8) circle (.4pt) -- (0,0);
\filldraw (.5,0) circle (.4pt) -- (.75, .4) circle (.4pt) -- (.25, .4) circle (.4pt) -- (.5,0);

\draw [fill= black](.1,0) circle (.4pt) -- (.25,.4) circle (.4pt) -- (.2,0) circle (.4pt) -- (.25,.4) circle (.4pt) -- (.3,0) circle (.4pt) -- (.25,.4) circle (.4pt) -- (.4,0) circle (.4pt);
\draw [fill= black] (.45,.72) circle (.4pt) -- (.6,.64) circle (.4pt) -- (.4,.64) circle (.4pt) -- (.75,.4) circle (.4pt) -- (.3,.48) circle (.4pt);
\draw [fill= black] (.75,0) circle (.4pt) -- (.75,.4) circle (.4pt);
\draw [fill= black] (.75,0) -- (.85,.24) circle (.4pt);
\draw [fill= black] (.75,0) -- (.9,.16) circle (.4pt);
\end{tikzpicture} 
\caption{A triangulation obtained from the triforce by a series of conical facet refinements.}
}
\end{figure}

\section{The internal edge graph}

The proofs of the main results of this paper all involve the \emph{internal edge graph}, formed by the edges of a triangulation $\Gamma$ that meet the interior of the simplex $\Delta$.  In this section, we study the properties of this graph when $d \geq 3$, and $\ell_1 = \ell_2 = 0$.

Let $f_i^j$ denote the number of $i$-simplices in $\Gamma$ that are carried by $j$-faces of $\Delta$.  From equation \eqref{eq:excess formula}, we see that $\ell_0 = 0$, $\ell_1 = f_0^{d-1}$, and
\begin{equation}\label{eqn:formula for l2}
\ell_2 = f_1^{d-1} -f_0^{d-2} - (d-1)f_0^{d-1}.
\end{equation}
In particular, if $\ell_1 = 0$ then $\ell_2 = f_1^{d-1} - f_0^{d-2}$.

Stanley proved non-negativity of all $\ell_i$, using methods from commutative algebra. 
It follows that if there are no interior vertices (i.e., if $\ell_1 = 0$) then 
$f_1^{d-1} \geq f_0^{d-2}$.  Note that every vertex carried by a $(d-2)$-face is contained 
in an interior edge, so this inequality is a statement about the internal edge graph. 
We now give a combinatorial proof of this inequality in a stronger form, showing in particular
that it holds separately on each connected component of the internal edge graph.

\begin{definition}
The \emph{internal edge graph} of a subdivision $\Gamma$ of $\Delta$ is the graph
$\Sigma(\Gamma)$ whose edges are the edges of $\Gamma$ carried by the improper face $\Delta$,
and whose vertices are the vertices incident to those edges. 
\end{definition}

\noindent When $\Gamma$ is clear from context, we will write $\Sigma$ rather than $\Sigma(\Gamma)$. 

\begin{proposition}\label{prop:each component has a low excess or low chi}
Let $\Gamma$ be a triangulation of a $(d - 1)$-simplex $\Delta$ with $d \geq 3$ and 
$\ell_1 = 0$.
Then each connected component $C$ of $\Sigma(\Gamma)$ contains either a vertex of excess
less than $d-2$ or a simple $3$-cycle. 

Furthermore, if $d \geq 4$, and $C$ has no vertices of excess less than $d-2$ then 
$C$ contains at least two distinct simple $3$-cycles. 
\end{proposition}

\begin{proof}
We may assume that $\Delta$ is the standard $(d-1)$-simplex in $\R^{d}$, i.e., 
the convex hull of the standard basis vectors $e_1,\ldots,e_{d}$. Let
$F_i = \Conv{e_1,\ldots,\widehat{e_i},\ldots,e_{d}}$, and let 
$\pi_i : \R^d \to \R$ be the projection onto the $i^{\mathrm{th}}$ coordinate axis.

Suppose $C$ is a connected component of $\Sigma$ all of whose vertices have excess $d-2$. 
We now show that $C$ contains a simple 3-cycle. 

Consider the two codimension 1 faces carrying the endpoints of an edge of $C$. Without 
loss of generality, we may assume these to be $F_1$ and $F_2$.  Let $L$ be the linear 
functional $ \pi_1 + \pi_2$.  Let $e= \Conv{v_1, v_2}$ be an edge of $C$ such that 
$v_1 \in F_1$, and $v_2 \in F_2$, such that $L(v_1 + v_2)$ is maximal among all such 
edges. 
For $(i,j) \in \{(1,2),(2,1)\}$, let $P_i : \R^{d}\to\R^d$ be the projection to the 
linear hyperplane spanned by $F_i$ along $v_j-v_i$. 
More explicitly, for all $w \in \R^{d}$,
\[
P_i(w) = w - \frac{\pi_i(w)}{\pi_i(v_j)}(v_j - v_i),
\]
Substituting in the definitions above, we observe that 
\begin{align*}
\pi_i(v_j)L(P_i(w) - v_i) & = \pi_i(v_j)\pi_j(P_i(w)) - \pi_i(v_j)\pi_j(v_i) \\
& = \pi_i(v_j)\pi_j(w) + \pi_j(v_i)\pi_i(w) - \pi_i(v_j) \pi_j(v_i).
\end{align*}
This is invariant under switching $i$ and $j$, and hence
\begin{equation}\label{e:projection equality}
\pi_1(v_2)L(P_1(w) - v_1) = \pi_2(v_1)L(P_2(w) - v_2).
\end{equation}

Applying $P_1$ to the closed star $\Cl(\Star(e))$, we see that $e$ projects to $v_1$, and 
2-faces project to segments ending at $v_1$. Since $e$ is an interior edge, $v_1$ is in 
the relative interior of $P_1(\Cl(\Star(e)))$.  Therefore,  
for any hyperplane $H$ in the image of $P_1$ that contains $v_1$, there is a vertex of 
the link of $e$ that projects to either side of that hyperplane. In particular,
there is a vertex $w$ in the link of $e$ such that $L(P_1(w) - v_1) > 0$. 
Recalling that all vertices of $C$ have excess $d-2$, we note that 
$\pi_1(v_2)>0$ and $\pi_2(v_1)>0$, otherwise they would be contained in 
$F_1\cap F_2$.
Therefore, by \eqref{e:projection equality}, $L(P_2(w) - v_2) > 0$. 

We claim that $w$ is not carried by $F_1$ or $F_2$. 
Indeed, if $w$ is carried by $F_1$, we would have $P_1(w) = w$ and 
$L(v_2 + w) > L(v_2 + v_1)$, contradicting our maximality assumption on $e$. Thus 
$w \notin F_1$. An identical argument shows that $w$ is not carried by $F_2$. This proves 
the claim.  Then without loss of generality, we may assume that $w$ is carried by $F_3$. 
The three interior edges of $C$ connecting $v_1$, $v_2,$ and $w$ form a simple 3-cycle. 
If $d = 3$, then we are done.

Suppose $d \geq 4$.  We must show that $C$ contains another simple $3$-cycle.  
Since $w$ is in $\lk e$, $\Gamma$ has an interior $2$-face $G = \Conv{v_1,v_2, w}$. 
Let $p$ denote the projection along $G$ onto the codimension 2 linear subspace spanned by 
$F_1 \cap F_2$. Then $p(G)$ is in the relative interior of $p(\Cl(\Star(G))$, and hence 
there exists a vertex $w' \notin F_1 \cap F_2 \cap F_3$ in the link of $G$. Without loss 
of generality, $w' \notin F_1$. Then $\Conv{v_1, w'}$ is an interior edge in $C$, and 
hence $w'$ is carried by a codimension 1 face of $\Delta$. Since $w' \notin F_2 \cap F_3$, 
we may assume, without loss of generality, that $w' \notin F_2$. Then the edges of $G$ 
and $\Conv{v_1, v_2, w'}$ form two distinct simple $3$-cycles in $C$. 
\end{proof}

Let $f_i^j(C)$ denote the number of $i$-faces of $C$ that are carried by $j$-faces of 
$\Delta$.

\begin{corollary} \label{cor:non-negative}
Suppose $d \geq 3$ and $\ell_1 = 0$.  Then each connected component $C$ of $\Sigma$ has
\[
f_0^{d-2}(C) \leq f_1^{d-1}(C).
\]
\end{corollary}

\begin{proof}
Recall that the Euler characteristic of a connected graph $C$ is 
\[
\chi(C) = \# \{ \mbox{vertices of $C$}\} - \#\{ \mbox{edges of $C$} \}.
\]
This is also equal to $1 - h^1(C)$, where $h^1$ denotes the first Betti number.  Hence $\chi(C) \leq 1$, with equality if and only if $C$ is a tree.  Moreover, $\chi(C) = 0$ if and only if $C$ contains a unique simple cycle.

Let $C$ be a connected component of the internal edge graph.  Then $$f_0^{d-2}(C) - f_1^{d-1}(C) = \chi(C) - f_0^{< d-2}(C).$$  By Proposition~\ref{prop:each component has a low excess or low chi}, either $\chi(C) = 1$ and $f_0^{< d-2}$ is positive, or $\chi(C) \leq 0$.
\end{proof}

In Section~\ref{sec:dim23}, we will repeatedly use the following structural description of the internal edge graph.

\begin{proposition}\label{prop:components are trees or have cycles}
Suppose $d \geq 3$, $\ell_1=\ell_2 = 0$, and $C$ is a connected component of $\Sigma$.  Then either
\begin{enumerate}
\item $C$ is a tree with a unique vertex of excess less than $d-2$, or
\item $d=3$, every vertex of $C$ has excess $d-2$, and $C$ has a unique simple 3-cycle.
\end{enumerate}
\end{proposition}

\begin{proof}
Since $\ell_1 = 0$, we can write $\ell_2$ as a sum over connected components of $\Sigma$
\begin{align*}
\ell_2 & = \sum_C f_1^{d-1}(C) - f_0^{d-2}(C),\\
& = \sum_C f_0^{< d-2}(C) - \chi(C).
\end{align*}
By Proposition~\ref{prop:each component has a low excess or low chi} and Corollary~\ref{cor:non-negative}, each summand is nonnegative, and the summand corresponding to $C$ is zero if and only if either $C$ is a tree with a unique vertex of excess less than $d-2$, or $d=3$, every vertex of $C$ has excess $d-2$, and $C$ contains a unique simple 3-cycle.
\end{proof}

Note that Theorem \ref{thm:internaledgegraph} follows immediately from Proposition \ref{prop:components are trees or have cycles}.

\begin{remark} \label{rem:connected}
Note that the internal edge graph of any triangulation of a $2$-simplex is connected. To see this, consider a polyhedral subdivision of a triangle with disconnected internal edge graph.  It must contain a polygon that meets two different components of the internal edge graph, and such a polygon has at least four edges: at least two segments of the boundary, and at least one edge from each component of the internal edge graph that it meets. In particular, it is not a triangulation.
\end{remark}

\begin{definition}
Let $F$ be a facet in a triangulation of the simplex $\Delta$.  
We say that $F$ is a \emph{pyramid} if there is some proper face of $\Delta$ that 
contains every vertex of $F$ except one.
\end{definition}

\begin{remark} \label{rem:pyramids}
Understanding when a facet of a triangulation is a pyramid is helpful for potential 
applications, e.g., to the monodromy conjecture for nondegenerate hypersurfaces, as 
mentioned in the introduction.  Note that, if $\Gamma'$ is a conical facet refinement of 
$\Gamma$ along $G$, then both $G$ and every facet of $\Gamma'_G$ is such a pyramid. Hence, 
it follows from Theorem~\ref{thm:dim3} that, if $d = 4$ and $\ell(\Gamma;x) = 0$, then every facet is 
a pyramid.  Similarly, it follows from Theorem~\ref{thm:dim2} that, if $d = 3$ and 
$\ell(\Gamma;x) = 0$, then at most one facet is not a pyramid. 
\end{remark}

The following proposition is not used in the remainder of the paper, but illustrates how 
the structure of the internal edge graph gives useful statements about which facets of a 
triangulation with vanishing local $h$-polynomial are pyramids.

\begin{proposition} \label{prop:pyramid}
Let $\Gamma$ be a triangulation of a $(d - 1)$-simplex $\Delta$ with $d \ge 4$. If $\ell_1= \ell_2 = 0$,
then every facet $G$ of $\Gamma$ that contains an interior edge is a pyramid.
\end{proposition}

\begin{proof}
Let $\Gamma$ be a facet of a triangulation of a $(d - 1)$-simplex $\Delta$ with $d \ge 4$, such that $\ell_1= \ell_2 = 0$.  Let $G$ be a facet of $\Gamma$ that contains an interior edge.  We assume that $G$ is not a pyramid, and will derive a contradiction. Let $e_1 = \Conv{v_1, v_2}$ be an interior edge in $G$. Let $C$ be the connected component of the interior edge graph $\Sigma(\Gamma)$ containing $e_1$. 

By Proposition~\ref{prop:components are trees or have cycles}, either $v_1$ or $v_2$ is carried by a codimension 1 face.
Without
loss of generality, $v_1$ is carried by a codimension 1 face $F_1$. Note that $v_2 \notin F_1$. Since $G$ is not a pyramid, it has a vertex $w_2 \notin F_1$ distinct from $v_2$. Then $e_2 = \Conv{v_1, w_2}$ lies in $C$.
By Proposition~\ref{prop:components are trees or have cycles}, either $v_2$ or $w_2$ is carried by a codimension 1 face. Without
loss of generality, $v_2$ is carried by a codimension 1 face $F_2$. If $w_2 \notin F_2$, then $\Conv{v_2, w_2}$
is an interior edge and $C$ contains a cycle, contradicting Proposition~\ref{prop:components are trees or have cycles}. 
Hence, $w_2 \in F_2$. Since $G$ is not a pyramid, it has a vertex $w_1 \notin F_2$ distinct from $v_1$.
Then $e_3 = \Conv{w_1, v_2}$ lies in $C$. If $w_1 \notin F_1$, then $\Conv{v_1, w_1}$
is an interior edge and  $C$ contains a cycle, contradicting Proposition~\ref{prop:components are trees or have cycles}. 
Hence $w_1 \in F_1$. By Proposition~\ref{prop:components are trees or have cycles}, 
either $w_1$ or $w_2$ is carried by a codimension 1 face, meaning that they cannot both 
be contained in some third codimension 1 face $F_3$. Therefore, 
$e_4 = \Conv{w_1, w_2}$ lies in $C$, and $e_1,e_3,e_4,e_2$ forms a 
cycle in $C$, contradicting Proposition~\ref{prop:components are trees or have cycles}.
\end{proof}

\section{Dimensions 2 and 3} \label{sec:dim23}

In this section we prove our main structural results, that all subdivisions with vanishing local $h$ can be obtained by a sequence of conical facet refinements from the trivial subdivision and the triforce subdivison in dimension 2, and from the trivial subdivision in dimension 3.  For both results our proof is by induction on the number of vertices or interior edges in the subdivision.  In the induction step, we identify a subcomplex that arises from a conical subdivision of a facet in a coarser subdivision.

Recall that the support of a subcomplex $\Gamma' \subset \Gamma$ is the union of the faces in $\Gamma'$.  It is a closed subset of $\Delta$.
The relative boundary of any closed subset $F \subset \Delta$ is the intersection of $F$ with the closure of its complement $\Delta \smallsetminus F$.  Note, in particular, that if $F$ is the support of a subcomplex $\Gamma' \subset \Gamma$, then the relative boundary of $F$ is the support of a subcomplex of $\Gamma'$.  When no confusion seems possible, we will refer to this subcomplex as the relative boundary of $\Gamma'$.

\begin{lemma}\label{coarsening-condition}
Let $\Gamma$ be a triangulation of a simplex $\Delta$ and let $\Gamma'\subseteq \Gamma$ be a subcomplex whose support is a simplex $F$ of dimension $d-1$.  Suppose that $\Gamma'$ induces the trivial subdivision on the relative boundary of $F$ in $\Delta$.
Then $\ell(\Gamma; x) - \ell(\Gamma'; x)$ has nonnegative coefficients. 
In particular, if $\ell(\Gamma;x)=0$, then $\ell(\Gamma';x)=0$. 
\end{lemma}

\begin{proof}
Since $\Gamma'$ does not nontrivially subdivide the relative boundary of $F$ in $\Delta$, we get a subdivision $\Gamma''$ of $\Delta$ by replacing $\Gamma'$ with $F$ itself. Note that $\Gamma$ is then a facet refinement of $\Gamma''$ along $F$, so by Proposition~\ref{additivity of local h},
we have $\ell(\Gamma;x)=\ell(\Gamma'';x)+\ell(\Gamma';x)$, and the lemma follows.
\end{proof}

\begin{remark}
Note that, in the proof of the lemma, $\Gamma$ is a conical facet refinement of $\Gamma''$  (resp. obtained from $\Gamma''$ by a sequence of conical facet refinements) if and only $\Gamma'$ is a conical facet refinement of $F$ (resp. obtained from the trivial subdivision of $F$ by a sequence of conical facet refinements).
\end{remark}

\subsection{Dimension 2}
We now apply the lemma to prove Theorem~\ref{thm:dim2}, which says that any triangulation of a 2-dimensional simplex with vanishing local $h$-polynomial is obtained from either the trivial subdivision or the triforce by a sequence of conical facet refinements.  

\begin{proof}[Proof of Theorem~\ref{thm:dim2}]
If $\Gamma$ has no interior edges, then it is the trivial subdivision; this case of the theorem is obvious.  We proceed by induction on the number of interior edges, and consider two cases, according to the possible structures of the internal edge graph given by  Proposition~\ref{prop:components are trees or have cycles}.  (Recall also that the internal edge graph is connected, by Remark~\ref{rem:connected}.)

\proofsection{Case 1:}{$\Sigma(\Gamma)$ is a tree that contains exactly one vertex of excess zero}  We claim that $\Gamma$ is obtained from the trivial subdivision of $\Delta$ by a sequence of conical facet refinements.  To see this, note that any interior edge that contains a vertex of excess zero divides $\Gamma$ into two triangular subcomplexes whose relative boundary is not subdivided, as shown.
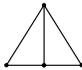
\begin{center}
\begin{tikzpicture}
\filldraw (0,0) circle (.6pt) -- (.5,0) circle (.6pt)-- (1,0)circle (.6pt) -- (.5,.8)circle (.6pt) -- (0,0);
\filldraw (.5,0)circle (.6pt) -- (.5,.8) ;
\node at (.5,-.2) {};
\node at (.5, 1) {};
\end{tikzpicture}
\end{center}
These two subcomplexes have vanishing local $h$-polynomials, by Lemma~\ref{coarsening-condition}. By induction, each is obtained from the trivial subdivision by a sequence of conical facet refinements.  Hence $\Gamma$ is obtained from the subdivision into two triangles by a sequence of conical facet refinements.  The subdivision into two triangles is itself a conical facet refinement of the trivial subdivision, and we conclude that $\Gamma$ is obtained from the trivial subdivision by a sequence of conical facet refinements, as claimed.

\proofsection{Case 2:}{$\Sigma(\Gamma)$ contains a simple 3-cycle and all of its vertices have excess 1} We claim that $\Gamma$ is obtained from the triforce by a sequence of conical facet refinements.  To see this, note that the three vertices of the simple 3-cycle must be carried by the three sides of $\Delta$.  This splits $\Gamma$ into four 
triangular subcomplexes, each with unsubdivided relative boundary, i.e., $\Gamma$ is a refinement of the triforce, and the induced subdivision on the relative boundary of each of the four triangles is trivial.  Hence, by Lemma \ref{coarsening-condition}, each of the four induced subdivisions has vanishing local $h$-polynomial.  In fact, the induced subdivision  of the interior triangle must be trivial (because $\Gamma$ has no interior vertices).  In each of the other three triangles, the induced subdivision cannot contain a cycle of interior edges, and hence the induced subdivisions are obtained from the trivial subdivision by a sequence of conical facet refinements.  We conclude that $\Gamma$ is obtained from the triforce by a sequence of conical facet refinements, as claimed.
\end{proof}

\subsection{Dimension 3}

We now apply similar inductive arguments to prove Theorem~\ref{thm:dim3}, which says that any subdivision $\Gamma$ of a 3-dimensional simplex $\Delta$ with vanishing local $h$-polynomial is obtained from the trivial subdivision by a sequence of conical facet refinements.

\begin{proof}[Proof of Theorem~\ref{thm:dim3}]
If $\Gamma$ has only 4 vertices, then it is the trivial subdivision and the conclusion is obvious.  We proceed by induction on the number of vertices.  By Proposition~\ref{prop:components are trees or have cycles}, each connected component of the internal edge graph $\Sigma(\Gamma)$ is a tree with exactly one vertex of excess less than 2. 

\proofsection{Case 1:}{$\Sigma(\Gamma)$ is empty}

Let $F$ be a 2-face of $\Delta$.  We claim that $\ell(\Gamma_F;x) = 0$.  To see this, note that any vertex carried by $F$ must be contained in an interior edge. Hence $\Gamma_F$ has no interior vertices, and $\ell_1(\Gamma_F;x) = 0$.  Now $\Gamma_F$ is a triangulation of the 2-simplex, so the symmetry of the local $h$-polynomial implies that $\ell_i(\Gamma_G;x) = \ell_{3-i}(\Gamma_F;x)$. Hence $\ell_2(\Gamma_F;x) = \ell_1(\Gamma_F;x) = 0$, and we conclude that $\ell(\Gamma_F;x) = 0$, as claimed.

\medskip

We proceed to consider three subcases, according to the internal edge graphs of the restrictions of $\Gamma$ to the 2-faces of $\Delta$.

\proofsection{Subcase 1.1:}{All 2-faces have empty internal edge graph}

Any interior 2-face of $\Gamma$ contains an edge that is carried by $\Gamma_F$ for some 2-face $F \leq \Delta$.  So $\Gamma$ has no interior 2-faces, and hence is the trivial subdivision. 

\proofsection{Subcase 1.2:}{Some 2-face has a 3-cycle in its internal edge graph}

Let $F_1$ be a 2-face of $\Delta$ such that $\Gamma_{F_1}$ contains a 3-cycle of interior edges. 
The vertices of this 3-cycle, $v_2$, $v_3$, and 
$v_4$, are carried by distinct edges of $F_1$.
Label the other 2-faces of $\Delta$ as $F_2,F_3,F_4$ so that $v_i$ is carried by $F_1\cap F_i$. 

The 2-face $\Conv{v_2,v_3,v_4}$ is contained in some 3-face $T=\Conv{v_2,v_3,v_4,x}$ of $\Gamma$. Since $\Sigma(\Gamma)$
is empty, the edges $\Conv{x,v_i}$ must be contained in the boundary of $\Delta$. Hence $x$ must lie in the intersection $F_2\cap F_3\cap F_4$, which is the vertex of $\Delta$ opposite $F_1$. Then the 2-faces $\Conv{v_1,v_2,x}$, 
$\Conv{v_2,v_3,x}$ and $\Conv{v_1,v_3,x}$ cut $\Delta$ into four tetrahedral regions (the cone over a triforce subdivision of $F_1$).  By Lemma \ref{coarsening-condition}, the restriction of $\Gamma$ to each of these regions has vanishing local $h$-polynomial.  By induction, the induced subdivision of each of these regions is obtained from the trivial subdivision by a sequence of conical facet refinements, and we are done.

\proofsection{Subcase 1.3:}{No 2-face has a 3-cycle, and some 2-face has an interior edge}

Let $F$ be a $2$-face of $\Delta$ such that $\Gamma_F$ contains an interior edge.  By Proposition~\ref{prop:components are trees or have cycles}, there is an edge $e$ carried by $F$ with a vertex of excess 0.  Label the vertices of $\Delta$ as $A$, $B$, $C$, and $D$, so that $F = \Conv{B,C,D}$ and $e$ contains $D$.  Let $E \in \Conv{B,C}$ be the second vertex of $e$.

Let $G$ be an interior 2-face of $\Gamma$ containing $e$. Then $G=\Conv{e,v}$ for some vertex $v$ of $\Gamma$. This vertex $v$ must lie on $\Conv{A,B}$ or $\Conv{A,C}$, so that the edges $\Conv{E,v}$ and $\Conv{D,v}$ are not interior.  Furthermore, $v\ne B,C$, since $G$ is interior.  Then either $v=A$, or, without loss of generality, we may assume $\sigma(v)=\Conv{A,C}$.  These possibilities are illustrated in Figure~\ref{fig:twocases}.

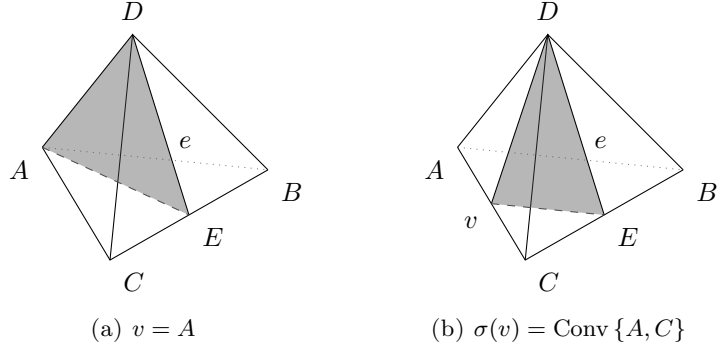
\begin{figure}[h!] 
\centering
\hfill
\begin{subfigure}{0.3\textwidth}
\centering
\begin{tikzpicture}[scale=3]
\coordinate (A) at (0,0.5); \node[below left=0.5ex of A] (Alab) {$A$};
\coordinate (B) at (1,0.4); \node[below right=0.5ex of B] (Blab) {$B$};
\coordinate (C) at (0.3,0); \node[below right=0.5ex of C] (Clab) {$C$};
\coordinate (D) at (0.4,1); \node[above=0.5ex of D] (Dlab) {$D$};
\path (B) -- (C) coordinate[midway] (E);
\node[below right=0.5ex of E] (Elab) {$E$};
\draw[dotted,color=white!50!black]
(A)
-- (B);
\draw
(A)
-- (C)
-- (B);
\draw 
(B)
-- (D)
-- (C);
\fill[color=white!30!black,opacity=0.4]
(A)
-- (E)
-- (D);
\draw 
(A)
-- (D);
\draw
(E)
-- (D) coordinate[pos=0.4] (ePos);
\draw[dashed,color=white!30!black]
(A)
-- (E);
\node[right=0.3ex of ePos] (e) {$e$};
\end{tikzpicture}
\caption{$v=A$}
\label{fig:v-equals-a}
\end{subfigure}
\hfill 
\begin{subfigure}{0.3\textwidth}
\centering
\begin{tikzpicture}[scale=3]
\coordinate (A) at (0,0.5); \node[below left=0.5ex of A] (Alab) {$A$};
\coordinate (B) at (1,0.4); \node[below right=0.5ex of B] (Blab) {$B$};
\coordinate (C) at (0.3,0); \node[below right=0.5ex of C] (Clab) {$C$};
\coordinate (D) at (0.4,1); \node[above=0.5ex of D] (Dlab) {$D$};
\path (B) -- (C) coordinate[midway] (E);
\node[below right=0.5ex of E] (Elab) {$E$};
\path (A) -- (C) coordinate[pos=0.5] (V1);
\node[below left=0.5ex of V1] (V1lab) {$v$};
\draw[dotted,color=white!50!black]
(A)
-- (B);
\draw
(A)
-- (C)
-- (B);
\draw 
(B)
-- (D)
-- (C);
\fill[color=white!30!black,opacity=0.4]
(V1)
-- (E)
-- (D);
\draw
(V1)
--(D);
\draw 
(A)
-- (D);
\draw
(E)
-- (D) coordinate[pos=0.4] (ePos);
\draw[dashed,color=white!30!black]
(V1)
-- (E);
\node[right=0.3ex of ePos] (e) {$e$};
\end{tikzpicture}
\caption{$\sigma(v)=\Conv{A,C}$}
\label{fig:v-not-equal-a}
\end{subfigure}
\hfill {}
\caption{Possible locations for $v$ in Subcase 1.3}
\label{fig:loc-for-v}
\label{fig:twocases}
\end{figure}

If $v = A$ then $G$, marked in grey in Figure \ref{fig:v-equals-a}, divides $\Delta$ into two distinct tetrahedra, that satisfy the conditions of the Lemma \ref{coarsening-condition}, and by the usual induction argument we are done.

We may therefore assume $\sigma(v) = \Conv{A,C}$ as in Figure \ref{fig:v-not-equal-a}. In this case, $G$ cuts off a tetrahedron $\Conv{v,C,D,E}$, which satisfies the conditions of Lemma \ref{coarsening-condition}. By induction, it suffices to consider the case where this tetrahedron is trivially subdivided. Also since $G$ is an interior 2-face, it must be contained in another tetrahedron $\Conv{v,D,E,w}$ of $\Gamma$. We must have $w\in \Conv{v,A}\cup\Conv{A,B}\cup \Conv{E,B}\setminus\set{v,E}$, in order to prevent the edges $\Conv{w,v}$ and $\Conv{w,E}$ from being interior, and so that $w$ is on the other side of $G$ from $C$.

Now if $\sigma(w)=\Conv{A,B}$, then $\Gamma_{\Conv{A,B,C}}$ will contain a cycle of interior edges, with vertices $v$, $E$, $w$. But we assumed that no 2-face has a cycle in its internal edge graph. Thus, $w\in\Conv{v,A}\setminus \set{v}$ or $w\in\Conv{E,B}\setminus\set{E}$. Without loss of generality, we may assume $w\in\Conv{E,B}\setminus\set{E}$. Then regardless of whether $w=B$ or is instead
in the interior of the segment, the tetrahedron $\Conv{v,D,w,C}$ is one to which Lemma \ref{coarsening-condition} applies. Moreover, $\Gamma_{\Conv{v,D,w,C}}$ is nontrivial, because $G$ lies in the interior. Hence $\Gamma_{\Conv{v,D,w,C}}$ is obtained from the trivial subdivision by a series of conical facet refinements, and the conclusion follows by induction.

\bigskip

\proofsection{Case 2:}{$\Sigma(\Gamma)$ is nonempty}

By Proposition~\ref{prop:components are trees or have cycles}, we know
that each connected component of $\Sigma(\Gamma)$ is a tree rooted at a unique vertex of excess less than $2$. We consider one of these components $C$ and let $v_1$ be one of its leaves at maximal distance from the root. Let $w$ be the unique vertex adjacent to $v_1$ in $C$, and let $F$ be the 2-face carrying $v_1$. Let $v_2, \cdots, v_r$ be the other leaves that are adjacent to $w$ and carried by $F$, as in Figure \ref{fig:possible-sigma}. 

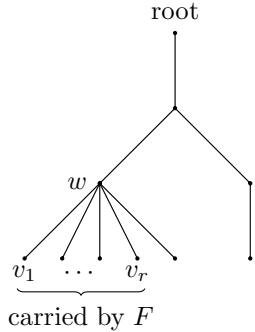
\begin{figure}[h!]
\centering
\begin{tikzpicture}
\filldraw (0,0) circle (.6pt)-- (0,-1)circle (.6pt) -- (-1,-2) circle (.6pt)-- (-2, -3)circle (.6pt);
\filldraw (-1.5, -3) circle (.6pt) -- (-1, -2) circle (.6pt)-- (-1, -3)circle (.6pt);
\filldraw (-.5, -3) circle (.6pt)--( -1, -2);
\filldraw (0,-1) -- (1, -2) circle (.6pt) -- (1, -3)circle (.6pt);
\filldraw (-1,-2) -- (0,-3) circle (.6 pt);
\node at (0, .3) {root};
\node at (-1.3, -2) {$w$};
\node at (-2, -3.2) {$v_1$};
\node at (-1.25, -3.2) {$\cdots$};
\node at (-.5, -3.2) {$v_r$};
\draw[decoration={brace},decorate](-.4,-3.4) -- (-2.1,-3.4) ;
\node at (-1.25,-3.8) {carried by $F$};
\end{tikzpicture}
\caption{Structure of $C$.}
\label{fig:possible-sigma}
\end{figure}

Let $X$ be the subcomplex of $\Gamma_F$ that is the union of the closed stars of $v_1, \ldots, v_r$. 

\proofsection{Claim:} {$X \subseteq \Cl(\Star(w))$.}

In other words, we claim that $\Gamma$ contains, as a subcomplex, the cone over $X$ with vertex $w$. To prove this, it suffices to show that if $G$ is a $k$-face of $\Gamma_F$ containing some vertex $v_i$, then $G$ is contained in a face of $\Star(w)$. Any such $G$ must be contained in an interior $(k+1)$-face of $\Gamma$.  Let $w'$ be a vertex such that $\Conv{G,w'}$ is an interior $(k+1)$-face. Note that $\Conv{v_i,w'}$ is an interior edge. Since $v_i$ is a leaf of the internal edge graph, we have $w'=w$.  Hence $\Conv{G,w}$ is a face of $\Star(w)$ that contains $G$.  This proves the claim.

\proofsection{Subcase 2.1:} {There is a vertex of $X$ carried by $F$ that is not in $\{ v_1, \ldots, v_r \}$}

Let $v$ be a vertex of $X$ that is carried by $F$ and not contained in $\{v_1, \ldots, v_r\}$. Since $v\in X$, there is an edge $\Conv{v,v_i}$ in $X$ for some index $i$. By the claim proved above, $\Conv{v,v_i,w}$ is a face of $\Gamma$, and $\Conv{v,w}$ is an edge of $C$. Since $v \not \in \{v_1, \ldots, v_r\}$, it must be the parent of $w$. 

Consider the link $\lk_{\Gamma_F}(v)$.  This is a cycle of edges in $\Gamma_F$ that contains $v_i$.  If every edge in $\lk_{\Gamma_F}(v)$ contains some $v_j \in \{ v_1, \ldots, v_r \}$, then $\Star(v)$ is contained in $\Cl(\Star(w))$ and hence $\Conv{v,w}$ is the unique interior edge containing $v$.  This contradicts the fact that $v$ is the parent of $w$.  Hence $\lk_{\Gamma_F}(v)$ contains multiple vertices not in $\{v_1, \ldots, v_r \}$.

Choose vertices $w_1$ and $w_2$ in $\lk_{\Gamma_F}(v)$ that are not in $\{v_1, \ldots, v_r\}$ so that every vertex in the path between $w_1$ and $w_2$ that contains $v_i$ is in $\{v_1, \ldots, v_r\}$. Note that, since $w_1$ and $w_2$ are adjacent to $w$ and not in $\{v_1, \ldots, v_r \}$, and since $v$ is the parent of $w$ in $C$, the edges $\Conv{w,w_i}$ must not be interior.  Similarly, $w$ cannot be the root of $C$. Hence $w$ is carried by a $2$-face that contains both $w_1$ and $w_2$. 

It follows that $w_1$ and $w_2$ lie on an edge of $F$. Thus the subcomplex with support $\Conv{v,w_1,w_2,w}$ satisfies the conditions of Lemma \ref{coarsening-condition}. 

\proofsection{Subcase 2.1.1:}{$\Conv{v,w_1,w_2,w}$ is trivially subdivided}
We claim that this is impossible.  Indeed, if $\Conv{v,w_1,w_2,w}$ is trivially subdivided, then $\lk_{\Gamma_F}(v)$ consists of the edge $\Conv{w_1, w_2}$ together with a sequence of edges connecting $w_1$ to $w_2$ in which all other vertices are contained in $\{ v_1, \ldots, v_r \}$.  Then every facet containing $v$ is a pyramid with base on $F$ and vertex at $w$.  Hence, $v$ is a leaf of $C$.  This is impossible, since we have already shown that $v$ is the parent of $w$. This proves the claim.

\proofsection{Subcase 2.1.2:}{$\Conv{v,w_1,w_2,w}$ is nontrivially subdivided}
In this case, by the induction hypothesis, $\Gamma_{\Conv{v,w_1,w_2,w}}$ is obtained from the trivial subdivision by a series of conical facet refinements, and the result follows by induction.

\proofsection{Subcase 2.2:}{Every vertex of $X$ carried by $F$ is in $\{v_1, \ldots, v_r\}$}

In this case, every edge in the boundary of $X$ connects two vertices that are in the boundary of $F$.  Then the closure $X'$ of any connected component of the relative interior of $X$ is a subcomplex with support a polygon whose relative boundary in $F$ is trivially subdivided. If the support of $X'$ is a triangle, then we can apply Lemma~\ref{coarsening-condition} for the subcomplex of $\Gamma$ given by the cone over $X'$ with vertex $w$, and be done by induction.

If $X'$ is not a triangle, then we show that $X'$ is contained in a larger subcomplex $Y \subset \Gamma_F$ such that 
\begin{enumerate}
\item The relative boundary of $Y$ in $F$ is trivially subdivided.
\item The cone over $Y$ with vertex $w$ is a subcomplex of $\Gamma$.
\item The support of $Y$ is a triangle.
\end{enumerate}
Once we find such a $Y$, then the conclusion of the theorem follows by Lemma~\ref{coarsening-condition} and induction on the number of vertices.

In order to find $Y$ as above, we begin with $Y_0=X'$, and construct an increasing sequence of subcomplexes
\[
Y_0 \subset Y_1 \subset \cdots \subset Y_n = Y
\] 
such that each $Y_i$ has properties $(1)$ and $(2)$, and the support of $Y$ is a triangle.

We describe the construction of $Y_{i+1}$ from $Y_{i}$ in terms of corner cutting edges.

\begin{definition}
A boundary edge $e$ of $Y_i$ is \emph{corner cutting} if its vertices are carried by edges of $F$, and $Y_i$ lies on the side of $e$ opposite from the corner that it cuts off.
\end{definition}

In our construction, if $Y_i$ is not a triangle then $Y_{i+1}$ has strictly fewer corner cutting edges than $Y_i$.  Since $Y_0$ has at most $3$ corner cutting edges, the procedure will terminate for some $n \leq 3$.

\begin{figure}[h!]
\begin{center}
\begin{tikzpicture}[scale = 3]
\filldraw [color=white!70!black] [decorate, decoration = {snake}] (.4,0) to [out = 110, in = 220] (.65,.56);
\filldraw[color=white!70!black] (.4,0) -- (.6,0) -- (.8, 0.32) -- (.65, .56) -- (.4, 0) ;
\node at (.6, .24) {$Y_i$};
\draw (0,0)  -- (1,0) circle (.6pt) -- (.5,.8) -- (0,0);
\filldraw (1,0) circle (.6pt);
\filldraw (.6,0) circle (.6pt) -- (.8, 0.32) circle (.6pt) ;
\node at (1.1,-.2) {$v$};
\node at (.6, -.2) {$w_2$};
\node at (.9, .4) {$w_1$};
\node at (.76, .1) {$e$};
\end{tikzpicture}
\end{center}
\caption{A corner cutting edge $e$ with vertices $w_1$ and $w_2$}
\label{fig:ccedge}
\end{figure}
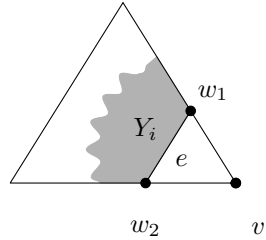

\proofsection{Subcase 2.2.1:}{$Y_i$ has no corner cutting edges} In this case, we claim that the support of $Y_i$ is triangular.  First, if $Y_i$ has no boundary edges, then its support is $F$. Suppose $Y_i$ has some boundary edge $e$. If both vertices of $e$ are carried by edges, then $Y_i$ is the on the same side as the vertex of $F$ cut off by $e$.  Any other boundary edge of $Y_i$ would have to be corner cutting.  Hence, the support of $Y_i$ is the triangle cut off by $e$, and we are done.

Otherwise, $e$ contains a vertex of $F$, and hence divides $F$ into two triangles.  Once again, any other boundary edge would have to be corner cutting, so the support of $Y_i$ is a triangle.  This proves the claim, and completes this subcase.

\bigskip

It remains to consider the situation where $Y_i$ has a corner cutting edge $e$, with vertices $w_1$ and $w_2$. Let $v$ be the vertex of $F$ cut off by $e$, as in Figure \ref{fig:ccedge}.  We consider two further subcases, according to whether or not $w$ is the root of $C$, as illustrated in Figure~\ref{fig:locations-for-w}.

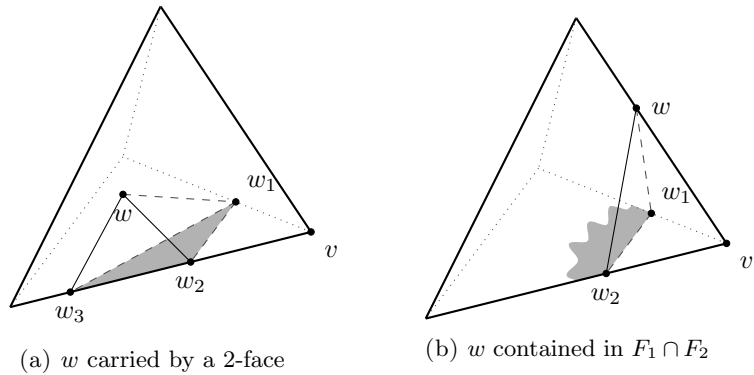
\begin{figure}[h!]
\begin{center}
\hfill
\begin{subfigure}{0.3\textwidth}
 \begin{tikzpicture}[scale=2]
 \coordinate (v_0) at (1.5,0); 
 \coordinate (v_a) at (.25,.5); 
 \coordinate (v_b) at (-.5,-.5); 
 \coordinate (top) at (.5, 1.5);
 \coordinate (w) at (0.25,0.25);
 
\path  (v_0) -- (v_a) coordinate[pos=.4] (w_1);
\path (v_0) -- (v_b) coordinate[pos=.4] (w_2);
\path (v_0) --(v_b) coordinate [pos=.8] (w_3);

\fill [white!70!black] (w_2)--(w_3)--(w_1) --(w_2);

\draw[dotted,color=white!30!black] (v_0) -- (v_a) ;
\draw  [thick] (v_0) -- (v_b) ;
\draw[dashed,color=white!30!black] (w_1) -- (w_2);
\draw[dotted,color=white!30!black] (v_a) -- (v_b);
\draw [thick] (top) -- (v_0);
\draw  [thick] (top) -- (v_b);
\draw[dotted,color=white!30!black] (top) -- (v_a);
\draw[dashed,color=white!30!black] (w) -- (w_1);
\draw (w) -- (w_3);
\draw[dashed,color=white!30!black] (w_1) -- (w_3);
\draw (w) -- (w_2);

\filldraw (v_0) circle (.6pt);
\filldraw (w_1) circle (.6pt);
\filldraw (w_2) circle (.6pt);
\filldraw (w) circle (.6pt);
\filldraw (w_3) circle (.6pt);

\node[below right=0.5ex of v_0]  {$v$};
\node[above right  =0.2ex of w_1]  {$w_1$};
\node[below =0.5ex of w_2]  {$w_2$};
\node[below =0.5ex of w]  {$w$};
\node[below =0.5ex of w_3]  {$w_3$};
\end{tikzpicture}
\caption{$w$ carried by a $2$-face}
\label{subfig:w3triangle}
\end{subfigure}
\hfill
\begin{subfigure}{0.3\textwidth}
 \begin{tikzpicture}[scale=2]
 \coordinate (v_0) at (1.5,0); 
 \coordinate (v_a) at (.25,.5); 
 \coordinate (v_b) at (-.5,-.5); 
 \coordinate (top) at (.5, 1.5);
 
\path  (v_0) -- (v_a) coordinate[pos=.4] (w_1);
\path  (v_0) -- (v_a) coordinate[pos=.5] (z_1);
\path (v_0) -- (v_b) coordinate[pos=.4] (w_2);
\path (v_0) -- (v_b) coordinate[pos=.5] (z_2);
\path (v_0) -- (top) coordinate[pos=.6] (w);

\filldraw [color=white!70!black] [decorate, decoration = {snake}] (z_2) to [out = 110, in = 220] (z_1);
\filldraw [color=white!70!black] (w_1) to (z_1) to  (z_2) to (w_2);

\draw[dotted,color=white!30!black] (v_0) -- (v_a) ;
\draw  [thick] (v_0) -- (v_b) ;
\draw[dashed,color=white!30!black] (w_1) -- (w_2);
\draw[dotted,color=white!30!black] (v_a) -- (v_b);
\draw [thick] (top) -- (v_0);
\draw  [thick] (top) -- (v_b);
\draw[dotted,color=white!30!black] (top) -- (v_a);
\draw[dashed,color=white!30!black] (w) -- (w_1);
\draw (w) -- (w_2);

\filldraw (v_0) circle (.6pt);
\filldraw (w_1) circle (.6pt);
\filldraw (w_2) circle (.6pt);
\filldraw (w) circle (.6pt);

\node[below right=0.5ex of v_0]  {$v$};
\node[above right  =0.1ex of w_1]  {$w_1$};
\node[below =0.5ex of w_2]  {$w_2$};
\node[right =0.5ex of w]  {$w$};

\end{tikzpicture}
\caption{$w$ contained in $F_1\cap F_2$}
\label{subfig:w-on-edge}
\end{subfigure}
\hfill{}
\caption{Possible locations for $w$, with the region $Y_i$ shaded}
\label{fig:locations-for-w}
\end{center}
\end{figure}

\proofsection{Subcase 2.2.2:}{$Y_i$ has a corner cutting edge, and $w$ is not the root of $C$}

If $w$ is not the root, then it is carried by some 2-face of $\Delta$ other than $F$. It cannot be carried by the face opposite $v$, since then both of the edges $\Conv{w_1,w}$ and $\Conv{w_2,w}$ would be interior, meaning that $w_1$ and $w_2$ would both be low excess vertices on the same component of the internal edge graph. Therefore without loss of generality, we can assume that $w$ is carried by the 2-face $F' \neq F$ containing $w_2$.

Note that the support of $Y_i$ is a polygon, and if we now think of its boundary not in $F$, but in the plane containing $F$, there are exactly two vertices of the boundary adjacent to $w_1$; one of these is $w_2$, and the other is some vertex, which we call $w_3$.

The edge $\Conv{w,w_3}$ cannot be interior, since $w_3$ is contained in the boundary of $F$, and hence has low excess, but $w_1$ is the root of $C$. Thus $w_3 \in F\cap F'$, as shown in Figure \ref{subfig:w3triangle}. It follows that the polygon $Y_i$ both contains and is contained in the triangle bounded by $\Conv{w_1, w_2}$,  $\Conv{w_2, w_3}$, and $\Conv{w_1, w_3}$.  Therefore, $Y_i$ is the triangle $\Conv{w_1, w_2, w_3}$, and we are done.

\proofsection{Subcase 2.2.3:}{$Y_i$ has a corner cutting edge, and $w$ is the root of $C$}

In this case, $w$ cannot have an interior edge to either $w_1$ or $w_2$, since $w_1$ and $w_2$ both have low excess, and $w$ is the root of $C$. Therefore, if $F_1$ and $F_2$ are the faces other than $F$ containing $w_1$ and $w_2$ respectively, then $w$ must lie on the edge $F_1\cap F_2$, possibly at the vertex, though this will not matter. This situation is illustrated in Figure \ref{subfig:w-on-edge}. Since $\Conv{w,w_1,w_2}$ is in fact a 2-face of $\Gamma$, the subcomplex with support $\Conv{w,w_1,w_2,v}$ satisfies the hypotheses of Lemma \ref{coarsening-condition}. Thus we may assume that it is trivially subdivided. We then set $Y_{i+1}$ to be the union of $Y_i$ with the triangle $\Conv{w_1, w_2, v}$.  Then $Y_{i+1}$ satisfies properties (1) and (2), and has fewer corner cutting edges than $Y_i$.  Repeating this procedure, as $Y_0$ has at most three corner cutting edges, we eventually arrive at $Y_n$ satisfying properties (1)-(3), and the theorem follows.
\end{proof}

\section{Higher dimensions} \label{sec:higherdim}

While Theorem~\ref{thm:internaledgegraph}, which is an immediate consequence of Proposition~\ref{prop:components are trees or have cycles}, says that the structure of the internal edge graph is essentially the same in higher dimensions, it seems that this is not sufficient for a useful classification of triangulations with vanishing local $h$-polynomial.  Indeed, there are infinitely many different triangulations of the 4-simplex with empty internal edge graph.  For instance, the join of the triforce with an arbitrary subdivision of the 1-simplex has this property.  These triangulations give rise to the following negative result.

\begin{proposition}
There is no finite collection of triangulations of the $4$-simplex with vanishing local $h$-polynomials from which all others can be obtained by a series of conical facet refinements.
\end{proposition}

\begin{proof}
Let $\Gamma_n$ be the subdivision of the $1$-simplex with $n$ interior vertices, and let $\Gamma'$ be the triforce.  If $e$ is an edge of $\Gamma_n$ such that both endpoints are interior, then the join of $e$ with the center facet of $\Gamma'$ is not a pyramid.  Hence $\Gamma_n * \Gamma'$ is a triangulation with vanishing local $h$-polynomial that has $n-1$ non-pyramid facets.  The proposition follows, since, as discussed in Remark~\ref{rem:pyramids}, all of the new facets introduced by a conical facet refinement are pyramids.
\end{proof}

\noindent We leave the problem of classifying triangulations with vanishing local $h$-polynomial in dimensions 4 and higher open, for future research.

\bibliography{math}
\bibliographystyle{amsalpha}

\end{document}